\documentclass[a4paper]{amsart}

\usepackage[T1]{fontenc}
\usepackage{amsmath, amssymb, amsthm, mathrsfs, mathtools, marvosym, graphicx}
\usepackage{varioref}
\usepackage[backref=page]{hyperref}
\usepackage{enumitem, xspace, ifthen, comment}
\usepackage[all]{xy}
\usepackage[dvipsnames]{xcolor}
\usepackage{appendix}
\usepackage[nameinlink, capitalize]{cleveref}
\usepackage{tikz}
\usetikzlibrary{shapes, calc, decorations.pathreplacing}

\setlist[enumerate]{
  label=(\thethm.\arabic*),
  before={\setcounter{enumi}{\value{equation}}},
  after={\setcounter{equation}{\value{enumi}}},
  itemsep=1ex
}

\setlist[itemize]{
  leftmargin=*,
  itemsep=1ex,
  label=$\circ$
}

\newtheorem*{thm-plain}{Theorem}
\newtheorem{thm}{Theorem}[section]
\newtheorem{lem}[thm]{Lemma}
\newtheorem{prp}[thm]{Proposition}
\newtheorem{cor}[thm]{Corollary}

\newtheorem{ques}[thm]{Question}
\numberwithin{equation}{thm}

\theoremstyle{definition}
\newtheorem{dfn}[thm]{Definition}

\newtheorem*{dfn-plain}{Definition}

\theoremstyle{remark}

\newtheorem{rem}[thm]{Remark}

\newtheorem{exm}[thm]{Example}
\newtheorem*{rem-plain}{Remark}

\newcommand{\inv}{^{-1}}
\newcommand{\from}{\colon}

\newcommand{\lto}{\longrightarrow}

\newcommand{\isom}{\cong}
\newcommand{\defn}{\coloneqq}

\newcommand{\wt}{\widetilde}
\newcommand{\wb}{\overline}

\renewcommand{\d}{\mathrm d}
\newcommand{\del}[1]{\partial_{#1}}

\newcommand{\ddual}{^{\smash{\scalebox{.7}[1.4]{\rotatebox{90}{\textup\guilsinglleft} \hspace{-.5em} \rotatebox{90}{\textup\guilsinglleft}}}}}

\newcommand{\factor}[2]{\left. \raise 2pt\hbox{$#1$} \right/\hskip -2pt \raise -2pt\hbox{$#2$}}

\newdir{ ir}{{}*!/-5pt/@^{(}} 
\newdir{ il}{{}*!/-5pt/@_{(}} 

\newcommand{\set}[1]{\left\{ #1 \right\}}

\def\rd#1.{\lfloor{#1}\rfloor}
\def\rp#1.{\lceil{#1}\rceil}
\def\tw#1.{\langle{#1}\rangle}

\renewcommand{\O}[1]{\mathscr{O}_{#1}}
\newcommand{\Omegap}[2]{\Omega_{#1}^{#2}}
\newcommand{\Omegar}[2]{\Omega_{#1}^{[#2]}}
\newcommand{\Omegal}[3]{\Omega_{#1}^{#2} \big( \!\log #3 \big)}

\newcommand{\T}[1]{\mathscr{T}_{#1}}
\newcommand{\can}[1]{\omega_{#1}}

\newcommand{\Reg}[1]{{#1}_{\mathrm{reg}}}
\newcommand{\Sing}[1]{{#1}_{\mathrm{sg}}}

\def\Hnought#1.#2.{\mathit{\Gamma} \!\left( #1, #2 \right)}
\def\HH#1.#2.#3.{\mathrm{H}^{#1} \!\left( #2, #3 \right)}
\def\euler#1.#2.{\chi \!\left( #1, #2 \right)}
\def\HHbig#1.#2.#3.{\mathrm{H}^{#1} \!\big( #2, #3 \big)}
\def\hh#1.#2.#3.{h^{#1} \!\left( #2, #3 \right)}
\def\RR#1.#2.#3.{R^{#1} #2_* #3}
\def\HHc#1.#2.#3.{\mathrm{H}_{\mathrm{c}}^{#1} \!\left( #2, #3 \right)}
\def\Hh#1.#2.#3.{\mathrm{H}_{#1} \!\left( #2, #3 \right)}
\def\Hom#1.#2.#3.{\mathrm{Hom}_{#1} \!\left( #2, #3 \right)}
\def\sHom#1.#2.{\mathscr{H}\!om \!\left( #1, #2 \right)}
\def\Ext#1.#2.#3.{\mathrm{Ext}^{#1} \!\left( #2, #3 \right)}
\def\sExt#1.#2.#3.{\mathscr{E}\!xt^{#1} \!\left( #2, #3 \right)}

\newcommand{\A}[1]{\mathbb A^{#1}}
\newcommand{\F}[1]{\mathbb F_{#1}}

\newcommand{\kahler}{K{\"{a}}hler\xspace}

\DeclareMathOperator{\Spec}{Spec}
\DeclareMathOperator{\Autn}{Aut^\circ}

\DeclareMathOperator{\Exc}{Exc}

\newcommand{\germ}[2]{\left( #1 \in #2 \right)} 

\newcommand{\et}{Regular Extension Theorem\xspace}
\newcommand{\lext}{Logarithmic Extension Theorem\xspace}
\newcommand{\lz}{Lipman--Zariski conjecture\xspace}
\newcommand{\lc}{log canonical\xspace}

\renewcommand{\theta}{\vartheta}
\renewcommand{\phi}{\varphi}

\newcommand{\C}{\ensuremath{\mathbb C}}

\renewcommand{\P}{\ensuremath{\mathbb P}}

\renewcommand{\frm}{\mathfrak m}

 \newcommand{\sB}{\mathscr B}

 \newcommand{\cQ}{\mathcal Q}

\definecolor{forrest}{RGB}{81,133,49}
\definecolor{mydarkblue}{RGB}{10,92,153}

\title[The Violation of the Lipman--Zariski conjecture]{The Violation of the Lipman--Zariski conjecture in positive characteristic}

\author{Patrick Graf}
\address{Lehrstuhl f\"ur Mathematik I, Universit\"at Bayreuth, 95440 Bayreuth, Germany}
\email{\href{mailto:patrick.graf@uni-bayreuth.de}{patrick.graf@uni-bayreuth.de}}
\urladdr{\href{http://www.pgraf.uni-bayreuth.de/en/}{www.graficland.uni-bayreuth.de}}

\date{May 6, 2022}

\hypersetup{
  pdfauthor={Patrick Graf},
  pdftitle={The Violation of the Lipman--Zariski conjecture in positive characteristic},
  pdfkeywords={Lipman--Zariski conjecture, positive characteristic, rational double points, log canonical surface singularities},
  pdfstartview={Fit},
  pdfpagelayout={OneColumn},
  pdfpagemode={UseNone},
  linktoc=all,
  breaklinks,
  linkcolor=[RGB]{0 0 96},
  citecolor=[RGB]{96 0 0},
  urlcolor=[RGB]{0 96 0},
  colorlinks}

\begin{document}

\begin{abstract}
We study the failure of the Lipman--Zariski conjecture in positive characteristic.
For rational double points, the conjecture holds true except for a short finite list of exceptions.
For log canonical surface singularities, the conjecture continues to hold with the same list of exceptions under an additional tameness hypothesis.
In particular, among rational double points in characteristic $p \ge 7$ Lipman's counterexample is the only one, and the conjecture holds for all tame $F$-pure normal surface singularities.
\end{abstract}

\maketitle


\section{Introduction}

The LZ (= Lipman--Zariski) conjecture asserts that a complex algebraic variety $X$ with locally free tangent sheaf $\T X$ is necessarily smooth~\cite{Lip65}.
Here $\T X = \mathscr{H}\!om_{\O X} \!\!\left( \Omegap X1, \O X \right)$ is the dual of the sheaf of \kahler differentials.
The case where~$X$ is a normal surface is essentially the only open case left, due to Lipman~\cite[Thm.~3]{Lip65}, Becker~\cite[Sec.~8, p.~519]{Becker78}, and Flenner~\cite[Corollary]{Flenner88}.
On the other hand, Lipman already observed that the conjecture fails in positive characteristic.
To be more precise, he gave the following (series of) counterexamples: consider the surface $X \subset \A3_k$ over a perfect field~$k$ of characteristic~$p$ defined by the equation
\[ xy - z^n = 0, \quad \text{where $p$ divides $n$.} \]
Then one can see that $\T X$ is freely generated by the two derivations $\delta_1 = x \del x - y \del y$ and $\delta_2 = \del z$, cf.~\cite[p.~892]{Lip65}.

The purpose of the present work is a more detailed study of the \lz on (normal) surfaces in positive characteristic.
First, in \cref{lz obs} below we give sufficient conditions for it to hold in a given situation.
We take our clue from the case $k = \C$.
In this case, the \lz is closely related to the problem whether $1$-forms on the smooth locus of $X$ extend to a (log) resolution of singularities $\pi \from Y \to X$ (with exceptional locus $E$).
In fact, we have the following chain of implications ($\dim X$ arbitrary, $k = \C$):

\[ \xymatrix{
\text{The sheaf $\pi_* \Omegal Y1E$ is reflexive} \ar@{=>}^{\text{\cite[Thm.~3.1]{GK13}}}[d] \\
\text{The sheaf $\pi_* \Omegap Y1$ is reflexive} \ar@{=>}^{\text{\cite[(1.6)]{SvS85}}}[d] \\
\text{The \lz holds for $X$}
} \]
\vspace{1ex}

If $X$ is a surface over a field of positive characteristic, neither of the above implications holds.
An example where the first implication fails is given in~\cite[Ex.~10.2]{FunnyExt}.
It has the bad behaviour that both $\pi_* \Omegal Y1E$ and $\pi_* \can Y$ are reflexive, but $\pi_* \Omegap Y1$ is not.
However, in~\cite[Thm.~1.3]{FunnyExt} we showed that the first implication remains valid if $X$ is \emph{tame} (cf.~\cref{dfn tame}).

The problem with the second implication is that, unlike in characteristic zero, a derivation $\delta$ on a variety $X$ may not lift to any resolution of~$X$, intuitively because it need not preserve the singular locus~$\Sing X$ -- or the Jacobian ideal of $X$, for that matter.
As an example, consider $\delta = \del z$ on Lipman's example above.
If the failure of this lifting property is ``not too bad'', we say that the MGR (= minimal good resolution) is \emph{almost equivariant}, cf.~\cref{dfn mgr equiv}.
With this notation, our first result is as follows.

\begin{prp}[Sufficient conditions for LZ to hold] \label{lz obs}
Let $\germ0X$ be a normal surface singularity over an algebraically closed field of characteristic $p > 0$, with $\pi \from Y \to X$ the MGR and $E = \Exc(\pi)$.
Assume that
\begin{enumerate}
\item\label{lz obs.1} the sheaf $\pi_* \Omegal Y1E$ is reflexive,
\item\label{lz obs.2} $\germ0X$ is tame, and
\item\label{lz obs.3} $\pi$ is almost equivariant.
\end{enumerate}
Then $\germ0X$ satisfies the \lz.
That is, if $\T X$ is free then $\germ0X$ is smooth.
\end{prp}

\begin{rem}
Of the above conditions,~\labelcref{lz obs.1} is probably the most restrictive, but in~\cite{FunnyExt} we gave sufficient conditions for it to hold.
Determining whether~\labelcref{lz obs.2} holds is easy from the dual graph.
The techniques of~\cite[\S2]{Wahl75} in principle allow to compute equivariance of $\pi$, and that paper also has some sufficient conditions.
Furthermore,~\cite{Hirokado19} has completely determined the almost equivariant rational double points.
\end{rem}

With \cref{lz obs} at hand, we investigate the following (vague) question:

\begin{ques} \label{lz ques}
Is Lipman's counterexample in any sense unique, or does the failure of the \lz rather occur ``generically''?
\end{ques}

Since Lipman's example is an RDP (= rational double point), we first concentrate on this class of singularities.
Here we obtain the following:

\begin{thm}[LZ conjecture for RDPs] \label{lz rdp}
Let $\germ0X$ be a two-dimensional RDP over an algebraically closed field $k$ of characteristic $p > 0$.
Then $\germ0X$ satisfies the \lz, except in the following cases:
\begin{align} \label{lz rdp list} \nonumber
\phantom{\hspace{4em}}
A_n, & \quad n \equiv -1 \!\!\!\mod p, \;\; \text{$p$ \textup{arbitrary,}} \\ \nonumber
D_n, & \quad n \ge 4, \;\; p = 2, \\
E_6, & \quad p = 2, 3, \\ \nonumber
E_7, & \quad p = 2, 3, \\ \nonumber
E_8, & \quad p = 2, 3, 5.
\end{align}
Conversely, in each of the above classes, there does exist an RDP that violates the \lz.
\end{thm}

The first item in the above list is precisely Lipman's example.
We see in particular that his example is indeed unique at least among RDPs in characteristic $p \ge 7$, providing a partial answer to \cref{lz ques}.

\begin{rem-plain}
In \cref{lz rdp}, we do not claim that \emph{every} RDP in~\labelcref{lz rdp list} violates the \lz.
For example in characteristic $p = 2$, the tangent sheaf of $E_8^0$, $E_8^1$ and $E_8^2$ is free, while for $E_8^3$ and $E_8^4$ it needs four generators (cf.~\cref{rdp list} for notation).
\end{rem-plain}

\begin{rem}[Fundamental groups]
The local fundamental groups of RDPs are known in all characteristics by~\cite{Durfee79} and~\cite{Artin77}.
Therefore, it is reasonable to ask whether the counterexamples in \cref{lz rdp} are exactly those for which the local fundamental groups are smaller than the corresponding characteristic zero group.
This is indeed true for the $A_n$ singularities, but not in general.
For example, all $E_8$ singularities ($p \le 5$) have ``too small'' fundamental group, but some of them do satisfy the \lz.
It \emph{is} true in general that if the fundamental group is the same as over \C, then the \lz is satisfied.
But in low characteristics, this applies only to a few cases.
\end{rem}

The equations of all RDPs over an algebraically closed field are known by~\cite{Artin77}.
Therefore, \cref{lz rdp} could also have been obtained by computer calculations and in fact, this is what we do in \cref{rdp list}.
But the point of having a conceptual argument is that it can also be applied in situations beyond RDPs, such as \lc singularities.
This is what we are going to do next.

\begin{thm}[LZ conjecture for lc surfaces] \label{lz lc}
Let $\germ0X$ be a \lc surface singularity over an algebraically closed field of characteristic $p > 0$.
Assume that $\germ0X$ is tame and that it is not an RDP contained in~\labelcref{lz rdp list}.
Then $\germ0X$ satisfies the \lz.
\end{thm}

\cref{lz lc sharp} shows that the statement of \cref{lz lc} is sharp in the sense that the tameness assumption cannot be dropped, even if $X$ is $F$-pure.
In terms of \cref{lz ques}, this suggests that if we do not impose a tameness condition, then Lipman's example (which itself is not tame) is very far from being unique.

Going back to the tame setting, by combining \cref{lz lc} with the calculations in \cref{rdp list}, we arrive at the following corollary.

\begin{cor}[LZ conjecture for tame lc surfaces] \label{cor tame lc}
Tame \lc surface singularities over an algebraically closed field satisfy the \lz, except for precisely seven RDPs of type $E_n$ in characteristic $p \le 5$:
\begin{itemize}
\item $E_6^0, E_8^0, E_8^1, E_8^2$ \textup{($p = 2$)},
\item $E_7^0, E_8^0$ \textup{($p = 3$)},
\item $E_8^0$ \textup{($p = 5$)}. \qed
\end{itemize}
\end{cor}

Since none of the above ``exceptional'' RDPs are $F$-pure, we also have:

\begin{cor}[LZ conjecture for tame $F$-pure surfaces] \label{cor tame F-pure}
The \lz holds for tame $F$-pure normal surface singularities over an algebraically closed field of characteristic $p > 0$. \qed
\end{cor}

It would be nice to have a direct proof of \cref{cor tame F-pure} that does not rely on any case-by-case analysis of explicit equations.
One might try to show that $F$-pure singularities satisfy the \lext and that they are almost equivariant, in order to apply \cref{lz obs}.
The first condition is probably true and could be proven using the classification results of~\cite{MehtaSrinivas91} and~\cite{Hara98}.
Unfortunately, the latter condition fails: the rational double point $D_{2n}^{n-1}$ ($p = 2$) is $F$-pure, but not almost equivariant.

\subsection*{Acknowledgements}

This work was begun, and completed to about 69\%, while the author stayed at the University of Utah in 2018/19, funded by a research fellowship of the DFG (= Deutsche Forschungsgemeinschaft).
I would like to thank the anonymous referee for her/his precise suggestions for improving the paper.

\section{Basic definitions}

By a \emph{(normal) surface singularity $\germ0X$ defined over a field $k$} we mean a scheme of the form $X = \Spec R$, where $(R, \frm)$ is a two-dimensional excellent (normal) local ring containing a field isomorphic to $k = \factor R\frm$.
By~\cite[Thm.]{Lipman78}, a normal surface singularity admits a resolution, and hence also a unique minimal good resolution.

\begin{dfn} \label{dfn tame}
Let $\germ0X$ be a normal surface singularity defined over a field~$k$.
We say that $\germ0X$ is \emph{tame} if for every resolution of singularities $\pi \from Y \to X$ with exceptional curves $E_1, \dots, E_\ell$, the determinant of the intersection matrix $(E_i \cdot E_j)$ is not divisible by $\operatorname{char} k$.
\end{dfn}

By the calculations in~\cite[Sec.~8.A]{FunnyExt}, it suffices to check tameness on a single resolution.
We also recall the following notation from~\cite{FunnyExt} in the special case of surface singularities.

\begin{dfn}[Extension Theorems] \label{dfn ext thm}
Let $\germ0X$ be a normal surface singularity defined over a field~$k$.
\begin{itemize}
\item We say that \emph{$\germ0X$ satisfies the \et (for $1$-forms)} if for some/any log resolution $\pi \from Y \to X$ with exceptional divisor $E = \Exc(\pi)$, the natural inclusion
\[ \pi_* \Omegap{Y/k}1 \xhookrightarrow{\quad} \Omegar{X/k}1 \defn \big( \Omegap{X/k}1 \big) \ddual \]
is an isomorphism.
Equivalently, the sheaf $\pi_* \Omegap{Y/k}1$ is reflexive.
\item We say that \emph{$\germ0X$ satisfies the \lext (for $1$-forms)} if for some/any $\pi$ as above, the natural inclusion
\[ \pi_* \Omegal{Y/k}1E \xhookrightarrow{\quad} \Omegar{X/k}1 \]
is an isomorphism.
Equivalently, the sheaf $\pi_* \Omegal{Y/k}1E$ is reflexive.
\end{itemize}
\end{dfn}

\begin{dfn}[Almost equivariant MGR] \label{dfn mgr equiv}
Let $\germ0X$ be a normal surface singularity defined over a field~$k$, with minimal good resolution $\pi \from Y \to X$.
Consider the natural short exact sequence of coherent sheaves
\[ 0 \lto \pi_* \T Y \lto \T X \lto \cQ \lto 0. \]
We say that $\pi$ is \emph{almost equivariant} if $\dim_\C \cQ_0 \le 1$.
Equivalently, and purely in terms of $Y$,
\[ \dim_\C \factor{\HH0.Y \setminus E.\T Y.}{\HH0.Y.\T Y.} \le 1, \]
where $E \subset Y$ is the exceptional locus of $\pi$.
\end{dfn}

The name stems from the case where $\cQ = 0$, or equivalently, when the sheaf $\pi_* \T Y$ is reflexive.
In this case one says that $\pi$ is \emph{equivariant}~\cite{Wahl75}.
This terminology comes from the observation that if $k = \C$ and $X$ is compact, then the identity component of the automorphism group $\Autn(X)$ will act on $Y$ in such a way that $\pi$ is equivariant.
For more information on this topic, see e.g.~\cite[Sec.~4]{GKK10}.

\section{The \lext for rational double points}

In this section, we determine which RDPs satisfy the \lext.
This is necessary for the application of \cref{lz obs} in the proof of \cref{lz rdp}.

\begin{thm}[\lext for RDPs] \label{lext rdp}
Let $\germ0X$ be a two-dimensional RDP over an algebraically closed field $k$ of characteristic $p > 0$.
Then $\germ0X$ satisfies the \lext, except in the following cases:
\begin{align} \label{lext rdp list} \nonumber
D_n, & \quad n \ge 4, \;\; p = 2, \\
E_6, & \quad p = 2, 3, \\ \nonumber
E_7, & \quad p = 2, 3, \\ \nonumber
E_8, & \quad p = 2, 3, 5.
\end{align}
Conversely, in each of the above classes, there does exist an RDP that violates the \lext.
\end{thm}

Note that not every singularity in~\labelcref{lext rdp list} violates the \lext.
E.g.~by \cref{D 2n n-1} below, the ($F$-pure) singularity $D_{2n}^{n-1}$ in characteristic $p = 2$ satisfies it.

\begin{proof}[Proof of \cref{lext rdp}]
If $p \ge 7$, the claim is a special case of~\cite[Thm.~1.2]{FunnyExt}.
Hence the only cases left are $A_n$ ($p = 2, 3, 5$), $D_n$ ($p = 3, 5$), and $E_{6,7}$ ($p = 5$).
For $A_n$, logarithmic extension immediately follows from~\cite[Thm.~6.1]{FunnyExt}.
For $D_n$ and $E_{6,7}$, the arguments from the proof of~\cite[Thm.~1.2]{FunnyExt} still apply because for the primes in question there does exist a ``tame resolution'', e.g.~for $D_n$ only $p = 2$ needs to be excluded.
Cf.~in particular cases~(7.8.6) and~(7.8.7) in that proof.

It remains to check the second half of the statement.
This is done in the following examples.
\end{proof}

\begin{exm}[$D_n$ singularities] \label{Dn sharp}
Fix a field $k$ of characteristic $p = 2$.
Then for the $D^0_{2n}$ singularity $X = \set{f = z^2 + x^2y + xy^n = 0} \subset \A3_k$, $n \ge 2$ arbitrary, the \lext does not hold.
More precisely, note that \kahler differentials on $X$ satisfy the relation $y^n \d x + (x^2 + nxy^{n-1}) \d y = 0$ and hence we may consider the (a priori only rational) $1$-form
\[ \sigma = y^{-n} \d y = \big( x^2 + nxy^{n-1} \big) \inv \d x. \]
As $y$ and $x^2 + nxy^{n-1}$ vanish simultaneously only at the origin, $\sigma$ is in fact a regular $1$-form on $X \setminus \set0 = \Reg X$.
In other words, $\sigma \in \HH0.X.\Omegar X1.$ is a reflexive differential form on $X$.
We blow up the origin $n-1$ times in a row, yielding a map $\phi \from \wt{\A{3}_k} \to \A3_k$.
In suitable coordinates, this map is given by
\[ \phi(u, v, w) = (uv^{n-1}, v, v^{n-1}w). \]
We compute
\[ \phi^*(f) = v^{2n-2}w^2 + u^2v^{2n-1} + uv^{2n-1} = v^{2n-2} \cdot \underbrace{\big( w^2 + uv(u + 1) \big)}_{\substack{\text{equation of strict} \\[.3ex] \text{transform $\wt X$ of $X$}}}. \]
We see that $\wt X$ can be parametrized rationally by the $(u, w)$-plane, namely by setting $v = \frac{w^2}{u(u + 1)}$.
In this parametrization, the pullback of $\sigma$ is given by
\[ \phi^*(\sigma) = v^{-n} \d v = \left( \frac{w^2}{u(u + 1)} \right)^{-n} \d \left( \frac{w^2}{u(u + 1)} \right) = \dots = \big( u(u + 1) \big)^{n-2} \frac{\d u}{w^{2n-2}}. \]
This shows that the extension of $\sigma$ to $\wt X$ does not have at most logarithmic poles.
Even worse, as $n$ goes to infinity, the pole orders become arbitrarily large.

Similar calculations for the $D^0_{2n + 1}$ singularities $\set{z^2 + x^2y + y^nz = 0}$ show that the reflexive form $\sigma = y^{-n} \d y = \big( x^2 + n y^{n-1}z \big) \inv \d z$ does not extend logarithmically.
\end{exm}

\begin{exm}[The $D_{2n}^{n-1}$ singularity] \label{D 2n n-1}
According to the SINGULAR program in \cref{compute mgs}, for the $D_{2n}^{n-1}$ singularity $X = \set{z^2 + x^2y + xy^n + xyz = 0}$ the tangent sheaf is freely generated by
\begin{align*}
v_1 & = y\del y + (x + z + ny^{n-1}) \del z, \\
v_2 & = x\del x + (z + y^{n-1}) \del z.
\end{align*}
Consider the $1$-forms $\alpha_1 = \d\log y \defn y\inv \d y$ and $\alpha_2 = \d\log x \defn x\inv \d x$.
Then we have $\alpha_i(v_j) = \delta_{ij}$ and hence $\set{\alpha_1, \alpha_2}$ is a basis of $\Omegar X1$.
(Note that e.g.~$\alpha_1$ does not actually have a logarithmic pole along $\set{y = 0}$ because $y$ vanishes to order two along that divisor.)
Clearly, the pullback of each $\alpha_i$ to any resolution has only logarithmic poles and so we see that $X$ satisfies the \lext.
\end{exm}

\begin{exm}[$E_n$ singularities] \label{En sharp}
Due to their similarity to \cref{Dn sharp}, we omit the calculations and only write down in each case the defining equation and a reflexive form that does not extend logarithmically.
We only treat $E_6$ and $E_7$ ($p = 2, 3$) since $E_8$ ($p = 2, 3, 5$) has already been done in~\cite[Ex.~10.1]{FunnyExt}.
\vspace{.5ex}
\begin{itemize}
\item $E_6^0$, $p = 2$: \; $f = z^2 + x^3 + y^2z$, \hspace{.345em} $\sigma = y^{-2} \d x = x^{-2} \d z$,
\item $E_7^0$, $p = 2$: \; $f = z^2 + x^3 + xy^3$, \; $\sigma = x \inv y^{-2} \d x = \big( x^2 + y^3 \big)\inv \d y$,
\item $E_6^0$, $p = 3$: \; $f = z^2 + x^3 + y^4$, \hspace{.85em} $\sigma = z \inv \d y = y^{-3} \d z$,
\item $E_7^0$, $p = 3$: \; $f = z^2 + x^3 + xy^3$, \; $\sigma = z \inv \d x = y^{-3} \d z$.
\end{itemize}
\end{exm}

\section{Proof of main results}

\subsection*{Proof of \cref{lz obs}}

Since $\germ0X$ is tame and satisfies the \lext, it also satisfies the \et by~\cite[Thm.~1.3]{FunnyExt}.
In other words, we have $\HH0.X.\Omegar X1. = \HH0.Y.\Omegap Y1.$.

Assume now that $\T X$ is freely generated by the two derivations $\set{v_1, v_2}$.
Since $\pi \from Y \to X$ is almost equivariant by assumption, the images of the $v_i$ in $\factor{\T X}{\pi_* \T Y}$ are linearly dependent, say $v_1 + \lambda v_2 = 0$ for some $\lambda \in \C$.
After replacing $v_1$ by $v_1 + \lambda v_2$, we may thus assume that $v_1$ lifts to $\wt v_1 \in \HH0.Y.\T Y.$.
Furthermore, in view of the short exact sequence
\[ 0 \lto \T Y(- \log E) \lto \T Y \lto \bigoplus_{i = 1}^\ell N_{E_i / Y} \lto 0 \]
and the negativity of the self-intersections $E_i^2 < 0$, it is clear that $\HH0.Y.\T Y. = \HH0.Y.\T Y(- \log E).$.

Let $\set{\alpha_1, \alpha_2}$ be the basis of $\Omegar X1$ dual to $\set{v_1, v_2}$.
By the \et, we can lift $\alpha_i$ to $\wt \alpha_i \in \HH0.Y.\Omegap Y1.$.
Then on $Y$, we still have $\wt \alpha_1(\wt v_1) \equiv 1$.
In particular, $\wt v_1$ has no zeroes.
But for each $i$, the derivation $\wt v_1$ restricts to a derivation $\wt v_1\big|_{E_i}$ on $E_i$, and the latter vanishes at each point of $(E - E_i) \cap E_i$, cf.~\cite[(1.10.2)]{Wahl75}.
Unless $E$ is empty, it follows that $E = E_1$ is smooth, and therefore an elliptic curve, since it carries the nowhere vanishing derivation $\wt v_1\big|_E$.
By~\cite[Prop.~2.12]{Wahl75}, $\pi$ is actually equivariant.
So we also have a lift $\wt v_2 \in \HH0.Y.\T Y(- \log E).$, and it satisfies $\wt\alpha_i(\wt v_j) \equiv \delta_{ij}$ on $Y$.

Pick an arbitrary (smooth) point $p \in E$.
Evaluating the above relation at $p$ shows that $\wt v_1(p), \wt v_2(p) \in T_p Y$ are linearly independent.
However, they are both contained in the one-dimensional subspace $T_p E \subset T_p Y$.
This contradiction shows that $E = \emptyset$, hence $\pi$ is an isomorphism and $\germ0X$ is smooth. \qed

\begin{lem}[Tame RDPs] \label{tame rdp}
Let $\germ0X$ be an RDP over an algebraically closed field $k$ of characteristic $p > 0$.
The cases where $\germ0X$ is not tame are exactly the following:
\begin{align} \label{det rdp list} \nonumber
\phantom{\hspace{4em}}
A_n, & \quad n \equiv -1 \!\!\!\mod p, \;\; \text{$p$ \textup{arbitrary,}} \\ \nonumber
D_n, & \quad n \ge 4, \;\; p = 2, \\ \nonumber
E_6, & \quad p = 3, \\ \nonumber
E_7, & \quad p = 2.
\end{align}
\end{lem}

\begin{proof}
The determinants of the intersection matrices of the exceptional curves in the minimal resolution are as follows:
\[ \begin{array}{lcl}
\big|\!\det(A_n)\big| & = & n + 1, \\[1ex]
\big|\!\det(D_n)\big| & = & 4, \\[1ex]
\big|\!\det(E_n)\big| & = & 9 - n.
\end{array} \]
The lemma follows immediately.
\end{proof}

\subsection*{Proof of \cref{lz rdp}}

Let $\germ0X$ be an RDP not contained in~\labelcref{lz rdp list}.
Then by \cref{lext rdp} and \cref{tame rdp}, the singularity $\germ0X$ satisfies the \lext and it is tame.
Also, the MGR of $\germ0X$ is equivariant by~\cite[Thm.~5.17]{Wahl75}.
Now apply \cref{lz obs} to get the first part of the statement.
Concerning the second part, e.g.~the following singularities have free tangent sheaf:
\begin{itemize}
\item $A_n$, \, $n \equiv -1 \!\mod p$, \, $p$ arbitrary,
\item $D_{2n}^0$ and $D_{2n + 1}^0$, \, $p = 2$,
\item $E_6^0$, $E_7^0$, $E_8^0$, \, $p = 2$,
\item $E_6^0$, $E_7^0$, $E_8^0$, \, $p = 3$,
\item $E_8^0$, \, $p = 5$.
\end{itemize}
For the verification of the above list, we refer to \cref{rdp list}. \qed

\subsection*{Proof of \cref{lz lc}}

Let $\germ0X$ be a \lc surface germ.
The possibilities for such $X$ have been classified in~\cite[Ch.~3]{Kollar13}, see also~\cite[Sec.~7.B]{FunnyExt}.
If we want to prove the \lz for a given $X$, this means that we are assuming $\T X$ to be free.
In particular, $K_X$ is Cartier and hence all discrepancies are integers.
This allows us to exclude many cases in the classification.

More precisely, in the minimal resolution $f \from X' \to X$, either all discrepancies $a_i = 0$, or all $a_i = -1$, by~\cite[Cor.~4.3]{KM98} (whose proof works in any characteristic).
\begin{itemize}
\item If all $a_i = 0$, then $\germ0X$ is an RDP and we have already handled this case in \cref{lz rdp}.
\item If all $a_i = -1$, then $\germ0X$ is either simple elliptic or a cusp.
(In case $\Exc(f)$ is a nodal rational curve, we need to blow up the node to get the minimal \emph{good} resolution $\pi \from Y \to X$.
The discrepancy of the new $(-1)$-curve is still $-1$.)
\end{itemize}
In the simple elliptic case, $E = \Exc(\pi)$ is an elliptic curve and in the cusp case, it is a cycle of smooth rational curves.
In both cases, $-(K_Y + E)$ is $\pi$-nef and the \lext for $\germ0X$ follows from~\cite[Thm.~6.1]{FunnyExt}.

In the simple elliptic case, $\pi$ is equivariant by~\cite[Prop.~2.12]{Wahl75}.
In the cusp case, we apply~\cite[Cor.~2.16]{Wahl75} instead (the first condition is vacuous since the set $\sB$ of ``rational boundary components'' of $E$ is empty, and a component $E_j$ as in the second condition exists because the fundamental cycle $Z_0$ satisfies $Z_0^2 < 0$).
Since $\germ0X$ is tame by assumption, \cref{lz obs} now applies to show that the \lz holds for $\germ0X$. \qed

\begin{exm}[Sharpness of \cref{lz lc}] \label{lz lc sharp}
Set $k = \F9$, and let $a \in k^\times$ be a generator with minimal polynomial $X^2 - X - 1$.
Consider the surface $X \subset \A3_{\wb k}$ defined by the equation $y^2z - x(x - az)(x + z) = 0$.
Thus $X$ is the cone over the elliptic curve $C \subset \P^2_{\wb k}$ defined by the same equation.
In particular, $\germ0X$ is log canonical and by the proof of \cref{lz lc}, it satisfies the \lext and the MGR is equivariant.
Also, $C$ is ordinary by~\cite[Ch.~IV, Prop.~4.21]{Har77} and therefore $\germ0X$ is $F$-pure by~\cite[Thm.~1.2]{MehtaSrinivas91}.
(Alternatively, one may use Fedder's criterion~\cite{Fedder83}.)
However, it is not tame because the exceptional curve $E \isom C$ has $E^2 = -3$.
By the SINGULAR program in \cref{compute mgs}, the tangent sheaf $\T X$ is freely generated by
\[ x \del x + y \del y + z \del z \quad \text{and} \quad a^5 y \del x + (x + a^6 z) \del y. \]
Hence $\germ0X$ violates the \lz.
It also follows that $\germ0X$ does not satisfy the \et (otherwise we could apply the arguments from the last paragraph of the proof of \cref{lz obs} to conclude that $X$ is smooth).
\end{exm}

\begin{appendices}

\crefalias{section}{appendix}

\section{Computations on all the rational double points} \label{rdp list}

In this appendix, we work through the list of all RDPs over an algebraically closed field~\cite{Artin77}.
We determine which of them are $F$-pure and satisfy the \lz, using the program in \cref{compute mgs}.
The information about $F$-purity is in principle contained in~\cite[Thm.~1.2 and Rem.~1.3]{Hara98}, albeit not in the exhaustive form presented here.
We also include the information whether the MGR is almost equivariant, which can easily be extracted from~\cite[Thms.~4.1 and~5.1]{Hirokado19}.
For the reader's convenience, we first list the ``classical'' equations from characteristic zero, which are also valid in sufficiently high characteristic:
\begin{equation*}
\begin{aligned}
A_n\from \quad & xy + z^{n+1} = 0 && \quad (n \ge 1) \\
D_n\from \quad & z^2 + x^2y + y^{n-1} = 0 && \quad (n \ge 4) \\
E_6\from \quad & z^2 + x^3 + y^4 = 0 \\
E_7\from \quad & z^2 + x^3 + xy^3 = 0 \\
E_8\from \quad & z^2 + x^3 + y^5 = 0 \\
\end{aligned}
\end{equation*}

\begin{table}[h!]
\begin{tabular}{|p{3em}|p{8em}|p{6.5em}|l|p{3.5em}|l|}
\hline
& Equation & Parameters & $F$-pure & almost equiv. & LZ holds \\ \hline \hline

$A_n$ & classical & \text{$n \ge 1$}, \text{$n \equiv -1 \!\!\mod p$} & \checkmark & \checkmark & --- \\ \hline
      &           & \text{$n \ge 1$}, \text{$n \not\equiv -1 \!\!\mod p$} & \checkmark & \checkmark & \checkmark \\ \hline
\end{tabular}
\caption{The $A_n$ singularities in characteristic $p \ge 2$}
\end{table}

\begin{table}[h!]
\begin{tabular}{|p{3em}|p{8em}|p{6.5em}|l|p{3.5em}|l|}
\hline
& Equation & Parameters & $F$-pure & almost equiv. & LZ holds \\ \hline \hline

$D_{2n}^0$ & $z^2 + x^2y + xy^n$ & $n \ge 2$ & --- & --- & --- \\ \hline
$D_{2n}^r$ & $z^2 + x^2y + xy^n + xy^{n-r}z$ & \text{$n \ge 2$}, \text{$1 \le r \le n - 2$} & --- & --- & --- \\ \hline
          &                                & \text{$n \ge 2$}, \text{$r = n - 1$} & \checkmark & --- & --- \\ \hline
$D_{2n+1}^0$ & $z^2 + x^2y + y^nz$ & $n \ge 2$ & --- & --- & --- \\ \hline
$D_{2n+1}^r$ & $z^2 + x^2y + y^nz + xy^{n-r}z$ & \text{$n \ge 2$}, \text{$1 \le r \le n - 2$} & --- & --- & \checkmark \\ \hline
            &                               & \text{$n \ge 2$}, \text{$r = n - 1$} & \checkmark & \checkmark & \checkmark \\ \hline \hline

$E_6^0$ & $z^2 + x^3 + y^2z$ & & --- & \checkmark & --- \\ \hline
$E_6^1$ & $z^2 + x^3 + y^2z + xyz$ & & \checkmark & \checkmark & \checkmark \\ \hline \hline

$E_7^0$ & classical & & --- & --- & --- \\ \hline
$E_7^1$ & $z^2 + x^3 + xy^3 + x^2yz$ & & --- & --- & --- \\ \hline
$E_7^2$ & $z^2 + x^3 + xy^3 + y^3z$ & & --- & --- & --- \\ \hline
$E_7^3$ & $z^2 + x^3 + xy^3 + xyz$ & & \checkmark & \checkmark & \checkmark \\ \hline \hline

$E_8^0$ & classical & & --- & --- & --- \\ \hline
$E_8^1$ & $z^2 + x^3 + y^5 + xy^3z$ & & --- & --- & --- \\ \hline
$E_8^2$ & $z^2 + x^3 + y^5 + xy^2z$ & & --- & --- & --- \\ \hline
$E_8^3$ & $z^2 + x^3 + y^5 + y^3z$ & & --- & \checkmark & \checkmark \\ \hline
$E_8^4$ & $z^2 + x^3 + y^5 + xyz$ & & \checkmark & \checkmark & \checkmark \\ \hline
\end{tabular}
\caption{The $D_n$ and $E_n$ singularities in characteristic $p = 2$}
\end{table}

\begin{table}[h!]
\begin{tabular}{|p{3em}|p{8em}|p{6.5em}|l|p{3.5em}|l|}
\hline
& Equation & Parameters & $F$-pure & almost equiv. & LZ holds \\ \hline \hline

$D_n$ & classical & $n \ge 4$ & \checkmark & \checkmark & \checkmark \\ \hline
\end{tabular}
\caption{The $D_n$ singularities in characteristic $p \ge 3$}
\end{table}

\begin{table}[h!]
\begin{tabular}{|l|p{8em}|l|p{3.5em}|l|}
\hline
& Equation & $F$-pure & almost equiv. & LZ holds \\ \hline \hline

$E_6^0$ & classical & --- & --- & --- \\ \hline
$E_6^1$ & $z^2 + x^3 + y^4 + x^2y^2$ & \checkmark & \checkmark & \checkmark \\ \hline \hline

$E_7^0$ & classical & --- & \checkmark & --- \\ \hline
$E_7^1$ & $z^2 + x^3 + xy^3 + x^2y^2$ & \checkmark & \checkmark & \checkmark \\ \hline \hline

$E_8^0$ & classical & --- & --- & --- \\ \hline
$E_8^1$ & $z^2 + x^3 + y^5 + x^2y^3$ & --- & \checkmark & \checkmark \\ \hline
$E_8^2$ & $z^2 + x^3 + y^5 + x^2y^2$ & \checkmark & \checkmark & \checkmark \\ \hline
\end{tabular}
\caption{The $E_n$ singularities in characteristic $p = 3$}
\end{table}

\begin{table}[h!]
\begin{tabular}{|l|p{8em}|l|p{3.5em}|l|}
\hline
& Equation & $F$-pure & almost equiv. & LZ holds \\ \hline \hline

$E_6$ & classical & \checkmark & \checkmark & \checkmark \\ \hline \hline

$E_7$ & classical & \checkmark & \checkmark & \checkmark \\ \hline \hline

$E_8^0$ & classical & --- & \checkmark & --- \\ \hline
$E_8^1$ & $z^2 + x^3 + y^5 + xy^4$ & \checkmark & \checkmark & \checkmark \\ \hline
\end{tabular}
\caption{The $E_n$ singularities in characteristic $p = 5$}
\end{table}

\begin{table}[h!]
\begin{tabular}{|l|p{8em}|l|p{3.5em}|l|}
\hline
& Equation & $F$-pure & almost equiv. & LZ holds \\ \hline \hline

$E_6$ & classical & \checkmark & \checkmark & \checkmark \\ \hline \hline

$E_7$ & classical & \checkmark & \checkmark & \checkmark \\ \hline \hline

$E_8$ & classical & \checkmark & \checkmark & \checkmark \\ \hline
\end{tabular}
\caption{The $E_n$ singularities in characteristic $p \ge 7$}
\end{table}

\section{Computing minimal generating sets} \label{compute mgs}

The following SINGULAR procedure takes as arguments a polynomial $f$ and a prime number $p$.
It then computes whether the singularity $\germ0X$ defined by $f$ in characteristic $p$ is $F$-pure and whether $\T X$ is free.
It also outputs a minimal generating set for $\T X$.

\begin{verbatim}
proc LipmanZariski(poly f, A, int p) {
    /** f: polynomial to be checked
        A: ring of which f is an element
           (should be char 0 polynomial ring)
        p: characteristic to be checked
           (needs to be a prime number) **/

    /* pass to char p polynomial ring */
    ring P = p, (x, y, z), ds; // local monomial ordering
    poly f = fetch(A, f);
    printf("f = %s in characteristic p = %s", f, p);

    /* check F-purity using Fedder's criterion */
    ideal Fed = x^p, y^p, z^p;
    if( reduce( f^(p-1), std(Fed) ) == 0 ) {
      print("X = { f = 0 } is not F-pure.");
    } else {
      print("X = { f = 0 } is F-pure.");
    }

    /* pass to quotient ring by f */
    qring R = std(f);
    poly f = fetch(P, f);

    /* compute minimal resolution of tangent sheaf */
    matrix Jac[1][3] = jacob(f);
    module T = syz(Jac);
    resolution rs = mres(T, 3);

    if( size(rs[1]) == 2 ) {
      print("T_X is free.");
    } else {
      print("T_X is not free.");
    }
    printf("Minimal generating set for T_X:");
    print(rs[1]);
}
\end{verbatim}

For example, in order to check the $E_8^0$ singularity $z^2 + x^3 + y^5 = 0$ in characteristic two, one may use the following commands in an interactive SINGULAR session (assuming the above code is contained in the file \texttt{funnylz.lib}):

\begin{verbatim}
    > LIB "funnylz.lib";
    > ring r = 0, (x, y, z), ds;
    > LipmanZariski(z^2 + x^3 + y^5, r, 2);

    f = z2+x3+y5 in characteristic p = 2
    X = { f = 0 } is not F-pure.
    T_X is free.
    Minimal generating set for T_X:
    0,y4,
    0,x2,
    1, 0
\end{verbatim}

In usual notation, this means that $\T X$ is generated by the two derivations $\del z$ and $y^4 \del x + x^2 \del y$.

\end{appendices}

\providecommand{\bysame}{\leavevmode\hbox to3em{\hrulefill}\thinspace}
\providecommand{\MR}{\relax\ifhmode\unskip\space\fi MR}
\providecommand{\MRhref}[2]{%
  \href{http://www.ams.org/mathscinet-getitem?mr=#1}{#2}
}
\providecommand{\href}[2]{#2}

\end{document}